\documentclass[reqno,11pt]{amsart} 
\usepackage{amsmath}
\usepackage{amssymb}
\usepackage{amsthm}
\usepackage{geometry}
\usepackage{enumerate}
\usepackage{tabularx}
\newgeometry{tmargin=3.45cm, bmargin=3.8cm, lmargin=2.7cm, rmargin=2.7cm}

\def\@begintheorem#1#2[#3]{%
  \pushQED{\qedthm}\deferred@thm@head{\the\thm@headfont \thm@indent
    \@ifempty{#1}{\let\thmname\@gobble}{\let\thmname\@iden}%
    \@ifempty{#2}{\let\thmnumber\@gobble}{\let\thmnumber\@iden}%
    \@ifempty{#3}{\let\thmnote\@gobble}{\let\thmnote\@iden}%
    \thm@swap\swappedhead\thmhead{#1}{#2}{#3}%
    \the\thm@headpunct
    \thmheadnl 
    \hskip\thm@headsep
  }%
  \ignorespaces}
\def\@endtheorem{\popQED\endtrivlist\@endpefalse }

\theoremstyle{definition}
\newtheorem{theorem}{Theorem}[section]
\newtheorem{lemma}[theorem]{Lemma}
\newtheorem{corollary}[theorem]{Corollary}
\newtheorem{proposition}[theorem]{Proposition}

\def\E{{\mathbb E}}
\def\R{{\mathbb R}}
\def\C{{\mathbb C}}
\def\P{{\mathbb P}}
\def\Z{{\mathbb Z}}

\def\Var{{\rm Var}}

\def\ep{\varepsilon}
\def\phi{\varphi}

\def\be{\begin{equation}}
\def\en{\end{equation}}
\def\bee{\begin{eqnarray*}}
\def\ene{\end{eqnarray*}}

\begin{document}

\title{Entropic CLT for smoothed convolutions \\
and associated entropy bounds
}

\author{Sergey G. Bobkov$^{1}$}
\thanks{1) 
School of Mathematics, University of Minnesota, Minneapolis, MN 55455 USA.
Research was partially supported by NSF grant DMS-1855575.}

\author{Arnaud Marsiglietti$^{2}$}
\thanks{2)
Department of Mathematics, University of Florida, Gainesville, FL 32611, USA}

\subjclass[2010]
{Primary 60E, 60F} 
\keywords{Central limit theorem, Entropic CLT}

\begin{abstract}
We explore an asymptotic behavior of entropies for sums of independent 
random variables that are convolved with a small continuous noise.
\end{abstract}

\maketitle

\section{Introduction}

Let $(X_n)_{n \geq 1}$ be independent, identically distributed (i.i.d.) 
random vectors in $\R^d$ with an isotropic distribution, that is,
with mean zero and an identity covariance matrix. 
By the central limit theorem (CLT), given a random vector $X$ in $\R^d$, 
independent of all $X_n$'s, the normalized sums
\be\label{(1.1)}
Z_n = \frac{1}{\sqrt{n}}\,(X + X_1 + \dots + X_n)
\en
are convergent weakly in distribution as $n \rightarrow \infty$ to the 
standard normal random vector $Z$ with density
\be\label{(1.2)}
\varphi(x) = \frac{1}{(2\pi)^{d/2}}\,e^{-|x|^2/2}, \qquad x \in \R^d.
\en
Suppose that $X$ has a finite second moment and an absolutely continuous
distribution, so that $Z_n$ have some densities $p_n$. A natural question 
of interest is whether or not this property (that is, the weak CLT)
may be strengthened as convergence of entropies
$$
h(Z_n) = -\int_{\R^d} p_n(x)\,\log p_n(x)\,dx
$$ 
to the entropy of the Gaussian limit $Z$. The usual entropic CLT corresponds 
to the i.i.d. case with $X=0$. Then, this CLT is known to hold,
if and only if $Z_n$ have densities $p_n$ with finite $h(Z_n)$ 
for some or equivalently all $n$ large enough 
\cite{Ba} 
(see also \cite{Mi}, \cite{A-B-B-N}, \cite{H-V}, \cite{J}, \cite{J-B}, 
\cite{B-C-G}, \cite{E}). 
What also seems remarkable, the presence of a small non-zero noise 
$X/\sqrt{n}$ in \eqref{(1.1)} may potentially enlarge the range of 
applicability of the entropic CLT. Here is one observation in this 
direction in terms of the characteristic function 
$$ 
f(t) = \E\,e^{i\left<t,X\right>}, \quad t \in \R^d. 
$$

\begin{theorem}\label{1.1}
If $f$ is compactly supported, and $X_1$ has a non-lattice distribution, 
then
\be\label{(1.3)}
h(Z_n) \rightarrow h(Z) \quad \mbox{as} \ n \rightarrow \infty.
\en
This convergence also holds for lattice distributions, if $f$ is supported on 
the ball $|t| \leq T$ for some $T>0$ depending on the distribution of $X_1$.
One may take $T = 1/\beta_3$, assuming that the $3$-rd absolute moment
$$
\beta_3 = \sup_{|\theta| = 1} \E\, |\left<X_1,\theta\right>|^3
$$ 
is finite.
\end{theorem}

The assumption of compactness on the support of the characteristic function of $X$ requires its density $p$ to be the restriction to $\R^d$ of an entire function on $\C^d$ of exponential type by Paley-Wiener theorems (cf. e.g. \cite{R}).

The entropic CLT \eqref{(1.3)} may equivalently be stated as the convergence 
$$
D(Z_n||Z) = \int_{\R^d}  p_n(x)\,\log \frac{p_n(x)}{\varphi(x)}\,dx 
 \ \rightarrow \ 0 \quad (n \rightarrow \infty)
$$
for the Kullback-Leibler distance (also called relative entropy or 
an informational divergence). It belongs to the family of so-called
strong (informational) distances, which dominate many other metrics
that are used in usual CLT's about the weak convergence of probability
distributions. As was mentioned to us by one of the referees, one
immediate consequence from \eqref{(1.3)} is the CLT for non-smoothed
normalized sums with respect to the Kantorovich transport distance $W_2$
(cf. Remark 4.4 for details).

In general, the hypothesis on the support of $f$ in Theorem \ref{1.1} 
cannot be removed,
but may be weakened by involving more delicate properties related to
the location of zeros of the characteristic function. This may be seen 
from the following characterization in one important example under
mild regularity assumptions on $f$.

\begin{theorem}\label{1.2}
Suppose that $X_1$ has a uniform distribution on the discrete cube 
$\{-1,1\}^d$, that is, with independent Bernoulli coordinates. Let 
the characteristic function $f$ of $X$ satisfy
\be\label{(1.4)}
\int_{\R^d} |f(t)|\,dt < \infty, \quad
\int_{\R^d} \frac{|f'(t)|}{\|t\|^{d-1}}\,dt < \infty,
\en
where $\|t\|$ denotes the distance from the point $t$ to the lattice 
$\pi \Z^d$. Then, the entropic CLT \eqref{(1.3)} holds true, 
if and only if
\be\label{(1.5)}
f(\pi k) = 0 \quad \mbox{for all} \ \ k \in \Z^d, \ k \neq 0.
\en
\end{theorem}

The second moment assumption on $X$ guarantees that $f$ has a bounded 
continuous derivative $f'(t) = \nabla f(t)$ with its Euclidean 
length $|f'(t)|$. The assumption of integrability of $f$ in \eqref{(1.4)} requires the density of $X$ to be continuous on $\R^d$. In dimension $d=1$, the condition \eqref{(1.4)} 
is fulfilled, as long as both $f$ and $f'$ are in $L^1$. 
If $d \geq 2$, \eqref{(1.4)} is more complicated, but is fulfilled,
for example, under decay assumptions such as
\be\label{(1.6)}
|f(t)| \, \leq \, \frac{c}{((1 + |t_1|) \dots (1 + |t_d|))^\alpha}, \qquad
|f'(t)| \, \leq \, \frac{c}{((1 + |t_1|) \dots (1 + |t_d|))^\alpha},
\en
holding for all $t = (t_1,\dots,t_d) \in \R^d$ with some constants
$\alpha > 1$ and $c>0$.

Although an information-theoretic meaning of the property \eqref{(1.5)} 
is not clear, it is indeed connected with the entropy functional $h(X)$.
Namely, under the conditions \eqref{(1.4)}-\eqref{(1.5)}, it turns out
that the entropy has to be non-negative. This is emphasized in the next 
statement, where we drop the isotropy condition and extend the
Bernoulli case to arbitrary integer valued random vectors. 
As before, we assume that $X$ is a continuous random vector in $\R^d$ 
with finite second moment, which is independent of all $X_n$'s.

\begin{theorem}\label{1.3}
Let $(X_n)_{n \geq 1}$ be a sequence of independent, integer valued 
random vectors, whose components have variance one. Then
$$ 
\limsup_{n \to \infty} \, h(Z_n) \, \leq \, h(X) + h(Z). 
$$
In particular, if $h(Z_n) \rightarrow h(Z)$ as 
$n \rightarrow \infty$, then necessarily $h(X) \geq 0$.
\end{theorem}

Actually, the independence assumption may further be weakened to the
uncorrelatedness (as explained in Theorem \ref{5.3} in the end of these notes).

We do not discuss here possible applications of the last conclusion 
in Theorem \ref{1.3}. Let us however stress that obtaining lower and upper 
bounds for the differential entropy, under various hypotheses or for 
different classes of probability distributions on the Euclidean space 
$\R^d$, is in itself an important and self-sufficient direction 
in information theory, which is motivated by many problems and is 
connected with other areas.
For example, applications of lower bounds to rate-distortion 
theory and channel capacity were put forward in \cite{M-K} 
(see also \cite{C-S-H}, \cite{H-J}, \cite{M-N-T}).
Let us also mention Bourgain's slicing problem 
in asymptotic geometric analysis, cf. \cite{Bou}. As a main conjecture, 
it states that, for any convex body $K$ in $\R^d$ there is
a hyperplane $H$ such that the $(d-1)$-dimensional volume of the slice
$H \cap K$ is bounded away from zero by a universal positive constant.
It was shown in \cite{B-Ma1} that the latter may equivalently be formulated 
as the property that, if $X$ is a random vector in $\R^d$ with
an isotropic log-concave distribution then
$$ 
h(X) \geq -cd 
$$
with some universal constant $c > 0$. Besides this conjecture, the past 
few years has seen a growing interest in the study of entropic inequalities 
as they shed new lights on fundamental problems in convex geometry 
(cf. e.g. \cite{B-Ma2}, \cite{C-W}, \cite{C-F-G-L-S-W}). 
We refer to the survey paper \cite{M-M-X} for further details on 
the connections between entropic inequalities and geometric and functional 
inequalities.

The paper is organized as follows. We start in Section \ref{2} with general upper and lower bounds on the Kullback-Leibler distance
\be\label{(1.7)}
D(X||Z) = \int_{\R^d}  p(x)\,\log \frac{p(x)}{\varphi(x)}\,dx
\en
from the distribution of $X$ to the standard normal law 
in terms of the $L^2$-distance 
\be\label{(1.8)}
\Delta = 
\|p-\varphi\|_2 = \bigg(\int_{\R^d} (p(x) - \varphi(x))^2\,dx\bigg)^{1/2}.
\en
Throughout, $Z$ denotes a standard normal random vector in $\R^d$, thus 
with density $\varphi$ as in \eqref{(1.2)} and with characteristic function
$$
g(t) = \E\,e^{i\left<t,Z\right>} = \int_{\R^d} e^{i\left<t,x\right>}\,
\varphi(x)\,dx = e^{-|t|^2/2}, \qquad t \in \R^d.
$$
As usual, the Euclidean space $\R^d$ is endowed with the canonical inner 
product $\left<\cdot,\cdot\right>$ and the norm $|\cdot|$. These bounds 
are applied in Section \ref{3} to express the entropic CLT as convergence of 
densities in $L^2$. Theorem \ref{1.1} and Theorem \ref{1.2} (in 
a somewhat refined form) are proved in Section \ref{4}. Using Proposition \ref{3.1},
the proofs employ recent results obtained in \cite{B-M} on local limit 
theorems with respect to the $L^2$ and $L^\infty$-norms. Theorem \ref{1.3} 
is proved in Section \ref{5}, where we also discuss the connection between 
entropy bounds and the entropic CLT.

\section{General bounds on relative entropy}\label{2}

Throughout this section, let $X$ be a random vector in $\R^d$ with 
density $p$, and let $\Delta$ be defined according to \eqref{(1.8)}.

\begin{proposition}\label{2.1}

Suppose that $\E\,|X|^2 = d$. If $\Delta \leq 1/e$, then
\be\label{(2.1)}
D(X||Z) \, \leq \, c_d\,\Delta\,\log^{\frac{d+4}{4}}(1/\Delta)
\en
with some constant $c_d>0$ depending on $d$ only. Moreover, if\, 
$\sup_x p(x) \leq M$ for some constant $M \geq (2\pi)^{-d/2}$, then
\be\label{(2.2)}
D(X||Z) \, \geq \, \frac{1}{2M}\,\Delta^2.
\en

\end{proposition}

First we collect a few elementary large deviation bounds.

\begin{lemma}\label{2.2}

For any $T \geq 1$,
\begin{itemize}
\item[(a)] $\int_{|x| \geq T}\, \varphi(x)\,dx \, \leq \, 
2d\,T^{d-2}\, e^{-T^2/2}$;
\item[(b)] $\int_{|x| \geq T}\, |x|^2\,\varphi(x)\,dx \, \leq \, 
2d\, T^d \, e^{-T^2/2}$.
\end{itemize}

\end{lemma}

\begin{proof}
Clearly, $(a)$ follows from $(b)$. To derive the second bound, write
\begin{eqnarray}\label{Gauss}
\E\,|Z|^2\,1_{\{|Z| \geq T\}} \, = \, \int_{|x| \geq T} |x|^2\, \varphi(x)\,dx  
 \, = \, \frac{d\omega_d}{(2\pi)^{d/2}}\, \int_T^\infty r^{d+1}\,e^{-r^2/2}\,dr,
\end{eqnarray}
where $\omega_d$ denotes the volume of the unit ball in $\R^d$.
Given $c > 1$, consider the function 
$$
u(T) = \int_T^\infty r^{d+1}\,e^{-r^2/2}\,dr - cT^d\,e^{-T^2/2}.
$$
We have $u(\infty) = 0$ and
$$
u'(T) = \big((c - 1)\,T^2 - cd\big)\, T^{d-1}\,e^{-T^2/2}.
$$
Thus, $u(T)$ is decreasing in some interval 
$0 \leq T < T_0$ and is increasing in $T \geq T_0$. 
Therefore, $u(T) \leq 0$ for all $T \geq 1$, if $u(1) = 0$, that is, for
\begin{eqnarray*}\label{c}
c = \sqrt{e} \int_1^\infty r^{d+1}\,e^{-r^2/2}\,dr.
\end{eqnarray*}
Using \eqref{Gauss}, we obtain
$$
c = \sqrt{e} \, \frac{(2\pi)^{d/2}}{d\omega_d} \E\,|Z|^2\,1_{\{|Z| \geq 1\}} 
\leq \sqrt{e} \, \frac{(2\pi)^{d/2}}{\omega_d},
$$
so
$$
\int_{|x| \geq T} |x|^2\, \varphi(x)\,dx  
 \, = \, 
\frac{d\omega_d}{(2\pi)^{d/2}} \left( u(T) + cT^d \, e^{-T^2/2} \right) 
 \, \leq \, \sqrt{e}\,d\, T^d\, e^{-T^2/2}.
$$
\end{proof}

To get the upper bound \eqref{(2.1)}, we also need to control
the weighted quadratic tails in terms of the $L^2$-distance $\Delta$. 

\begin{lemma}\label{2.3}

If $\E\,|X|^2 = d$, then for all $T \geq 1$,
$$
\int_{|x| \geq T} |x|^2\, p(x)\,dx \, \leq \,
2\,T^{\frac{d+4}{2}}\, \Delta + 2d\,T^d \, e^{-T^2/2}.
$$

\end{lemma}

\begin{proof}
We have
\bee
\int_{|x| \geq T} |x|^2\, p(x)\,dx 
 & = &  
d - \int_{|x| \leq T} |x|^2 p(x)\,dx \\
& = & 
\int_{|x| \leq T} |x|^2\,(\varphi(x) - p(x))\,dx  +
\int_{|x| \geq T} |x|^2\, \varphi(x)\,dx \\
& \leq & 
\int_{|x| \leq T} |x|^2\,|p(x) - \varphi(x)|\,dx  +
\int_{|x| \geq T} |x|^2\, \varphi(x)\,dx. \\
\ene
The last integral is
bounded by $2d\,T^d \, e^{-T^2/2}$. Also, by the Cauchy inequality,
$$
\bigg(\int_{|x| \leq T} |x|^2\,|p(x) - \varphi(x)|\,dx\bigg)^2 \, \leq \,
\int_{|x| \leq T} |x|^4\,dx \int_{\R^d} (p(x) - \varphi(x))^2\,dx
 \, = \,
\frac{d\omega_d}{d+4}\,T^{d+4} \,\Delta^2,
$$
where $\omega_d$ is the volume of the unit ball in $\R^d$.
Here, $\frac{d\omega_d}{d+4} < 4$.
\end{proof}

\begin{lemma}\label{2.4}

For all $T \geq 1$,
\begin{eqnarray}\label{(2.3)}
D(X||Z) 
 & \leq &
2d\,T^{d-2}\, e^{-T^2/2} +
(2\pi)^{d/2} \int_{|x| \leq T} (p(x) - \varphi(x))^2\, e^{|x|^2/2}\ dx 
\nonumber \\
 & & + \ 
\frac{2d-1}{2} \int_{|x| \geq T} 
|x|^2\, p(x)\, dx + \int_{|x| \geq T} p\log p\,dx. 
\end{eqnarray}

\end{lemma}

\begin{proof}
In definition \eqref{(1.8)}, we split the integration 
into the two regions. Using the inequality
$t \log t \leq (t-1) + (t-1)^2$, $t \geq 0$, and applying 
the first bound of Lemma \ref{2.2}, we have
\bee
\int_{|x| \leq T} \frac{p}{\varphi} \log \frac{p}{\varphi}\ \varphi\,dx
 & \leq &
\int_{|x| \leq T} \Big(\frac{p}{\varphi} - 1\Big)\, \varphi\,dx +
\int_{|x| \leq T} \Big(\frac{p}{\varphi} - 1\Big)^2\, \varphi\,dx \\
 & = & 
\int_{|x| \geq T} (\varphi - p)\, dx +
\int_{|x| \leq T} \frac{(p - \varphi)^2}{\varphi}\ dx \\
 & \leq &
2d\,T^{d-2}\,e^{-T^2/2} - \int_{|x| \geq T} p\, dx +
(2\pi)^{d/2} \int_{|x| \leq T} (p(x) - \varphi(x))^2\, e^{|x|^2/2}\, dx.
\ene
For the second region $|x| \geq T$, just write
\bee
\int_{|x| \geq T} p \log \frac{p}{\varphi}\, dx 
 & = &
\int_{|x| \geq T} p \log p\,dx \\
 & & + \
\frac{d}{2}\,\log (2\pi) \int_{|x| \geq T} p\,dx +
\frac{1}{2} \int_{|x| \geq T} |x|^2\, p(x)\, dx.
\ene
Combining these relations and noting that $\log(2\pi) < 2$, we thus get
\bee
D(X||Z) 
 & \leq &
2d\,T^{d-2}\, e^{-T^2/2} +
(2\pi)^{d/2} \int_{|x| \leq T} (p(x) - \varphi(x))^2\, e^{|x|^2/2}\ dx \\
 & & + \ 
(d-1) \int_{|x| \geq T} p(x)\, dx +
\frac{1}{2} \int_{|x| \geq T} |x|^2\, p(x)\, dx + \int_{|x| \geq T} p\log p\,dx. 
\ene
\end{proof}

As a consequence, we obtain:

\begin{lemma}\label{2.5}

For all $T \geq 1$,
$$
D(X||Z) \ \leq \ (2d+1)\,T^{d-1}\,e^{-T^2/2} +
\big((2\pi)^{d/2}+1\big)\,e^{T^2/2}\, \Delta^2 +
d\int_{|x| \geq T} |x|^2\, p(x)\, dx.
$$

\end{lemma}

\begin{proof}
We use the notation $a^+ = \max(a,0)$.
Subtracting $\varphi(x)$ from $p(x)$ and then adding, one can write
\bee
\int_{|x| \geq T} p\log p\,dx
 & \leq &
\int_{|x| \geq T} p(x)\log^+(p(x))\,dx \\
 & \leq &
\int_{|x| \geq T} |p(x) - \varphi(x)|\, 
\log^+(p(x))\,dx + \int_{|x| \geq T} \varphi(x)\, \log^+(p(x))\,dx.
\ene
Next, let us apply Cauchy's inequality together with the bound
$(\log^+(t))^2 \leq 4e^{-2}\, t$ so that to estimate the last integral 
from above by
$$
\bigg(\int_{|x| \geq T} \varphi(x)^2\,dx\bigg)^{1/2}
\bigg(\int_{|x| \geq T} \big(\log^+(p(x))\big)^2\,dx\bigg)^{1/2}
\leq \, 
\frac{2}{e} \, \bigg(\int_{|x| \geq T} \varphi(x)^2\,dx\bigg)^{1/2}.
$$
Here, using the first bound of Lemma \ref{2.2}, we have 
\bee
\int_{|x| \geq T} \varphi(x)^2\,dx 
 & = & 
\frac{1}{(4\pi)^{d/2}} \int_{|y| \geq T\sqrt{2}} \varphi(y)\,dy \\
 & \leq &
\frac{2d}{(4\pi)^{d/2}}\,(T\sqrt{2})^{d-2}\, e^{-T^2} \, < \, T^{d-1}
e^{-T^2}.
\ene
Therefore,
$$
\int_{|x| \geq T} p\log p\,dx \, \leq \, 
\int_{|x| \geq T} |p(x) - \varphi(x)|\, 
\log^+(p(x))\,dx + T^{\frac{d-1}{2}} e^{-T^2/2}.
$$
To simplify, the last integrand may be bounded by 
$$
\frac{1}{2}\,(p(x) - \varphi(x))^2 + \frac{1}{2}\,\big(\log^+(p(x))\big)^2
\leq \frac{1}{2}\,(p(x) - \varphi(x))^2 + 
\frac{1}{2}\, p(x),
$$
so,
$$
\int_{|x| \geq T} p\log p\,dx  \, \leq \, \frac{1}{2}\,\Delta^2 +
\frac{1}{2} \int_{|x| \geq T} p(x)\,dx + T^{\frac{d-1}{2}} e^{-T^2/2}.
$$
Using this estimate in \eqref{(2.3)} together with
$e^{|x|^2/2} \leq e^{T^2/2}$ for $|x| \leq T$, we get
\bee
D(X||Z) 
 & \leq &
2d\,T^{d-1}\, e^{-T^2/2} +
(2\pi)^{d/2}\,  e^{T^2/2}\int_{|x| \leq T} (p(x) - \varphi(x))^2\,dx \\
 & &
+ \ \frac{2d-1}{2}\, \int_{|x| \geq T} |x|^2\,p(x)\, dx + 
\frac{1}{2}\,\Delta^2 + 
\frac{1}{2} \int_{|x| \geq T} p(x)\,dx + T^{\frac{d-1}{2}} e^{-T^2/2}.
\ene
\end{proof}

\begin{proof}[{\bf Proof of Proposition} \ref{2.1}]
Combining Lemma \ref{2.5} with Lemma \ref{2.3}, we immediately get
$$
D(X||Z) \, \leq \,
(2d^2 + 2d+1)\,T^d\,e^{-T^2/2} + \big((2\pi)^{d/2}+1\big)\,e^{T^2/2}\, \Delta^2 +
2d\,T^{\frac{d+4}{2}}\,\Delta.
$$
To get \eqref{(2.1)}, it remains to take here 
$$
T = \sqrt{2 \log(1/\Delta) + \frac{d}{2}\,\log \log(1/\Delta)}.
$$ 
For the lower bound \eqref{(2.2)}, let us recall that $D(X||Z) = h(Z) - h(X)$.
By Taylor's expansion, for all $t \geq 0$ and $t_0 > 0$,
there is a point $t_1$ between $t$ and $t_0$ such that
$$
t\log t = t_0 \log t_0 + (\log t_0 + 1)(t - t_0) + \frac{(t-t_0)^2}{2 t_1}.
$$
Inserting $t = p(x)$, $t_0 = \varphi(x)$, we obtain a measurable 
function $t_1(x)$ with values between $p(x)$ and $\varphi(x)$, satisfying
$$
p(x)\log p(x) \, = \, \varphi(x) \log \varphi(x) +
(\log \varphi(x) + 1)\, 
(p(x) - \varphi(x)) + \frac{(p(x) - \varphi(x))^2}{2 t_1(x)}.
$$
Let us integrate this equality over $x$ and use $\E\, |X|^2 = d$ to get
$$
-h(X) = -h(Z) + \frac{1}{2} 
\int_{\R^d} \frac{(p(x) - \varphi(x))^2}{t_1(x)}\,dx.
$$
Hence
$$
D(X||Z) = \frac{1}{2} \int_{\R^d}
\frac{(p(x) - \varphi(x))^2}{t_1(x)}\,dx.
$$
It remains to use the assumptions $p(x) \leq M$ and $\varphi(x) \leq M$, 
so that $t_1(x) \leq M$ as well. 
\end{proof}

\section{Topological properties of relative entropy}\label{3}

Applying Proposition \ref{2.1} to a sequence of random vectors, we arrive 
at necessary and sufficient conditions for the convergence in the 
Kullback-Leibler distance $D$ in terms of the $L^2$-distances
$$
\Delta_n = 
\|p_n-\varphi\|_2 = \bigg(\int_{\R^d} (p_n(x) - \varphi(x))^2\,dx\bigg)^{1/2}.
$$
More precisely, we have:

\begin{proposition}\label{3.1}

Let $(Z_n)_{n \geq 1}$ be a sequence of random vectors in $\R^d$ 
with densities $p_n$. Suppose that as $n \rightarrow \infty$
\begin{itemize}
\item[(a)] $\E\, |Z_n|^2 \rightarrow d$;
\item[(b)] $\Delta_n \rightarrow 0$.
\end{itemize}
Then $D(Z_n||Z) \rightarrow 0$ or equivalently $h(Z_n) \rightarrow h(Z)$
as $n \rightarrow \infty$. Conversely, if $p_n$ are uniformly bounded, 
then the conditions $(a)-(b)$ are also necessary for the convergence in $D$.

\end{proposition}

Before turning to the proof, let us recall a basic abstract definition
of the Kullback-Leibler distance (i.e., relative entropy). Let $X$ and $Y$
be random elements in a measurable space $\Omega$ with distributions $\mu$ 
and $\nu$, respectively. If $\mu$ is absolutely continuous with 
respect to $\nu$ and has density $h = d\mu/d\nu$, 
the relative entropy of $\mu$ with respect to $\nu$ is defined as
$$
D(X||Y) \, = \, D(\mu||\nu) \, = \, 
\int_\Omega h\log h\,d\nu \, = \, \int p\, \log\frac{p}{q}\,d\lambda,
$$
where in the last equality we assume that $\mu$ and $\nu$ have densities 
$p$ and $q$ with respect to the dominating measure $\lambda$ on $\Omega$, 
so that $h = p/q$ (which is well-defined $\lambda$-almost everywhere). 
This definition does not depend on the choice of $\lambda$, and one may
always take $\lambda = \mu + \nu$, for example.
If $\mu$ is not absolutely continuous with respect to $\nu$, one puts 
$D(X||Y) = D(\mu||\nu) = \infty$. For basic properties of this functional,
we refer an interested reader to \cite{E-H}, and here only mention one
well-known relation
$$
\int_\Omega g\,d\mu \, \leq \, D(\mu||\nu) + \log \int_\Omega e^g\,d\nu.
$$
It holds for any measurable function $g$ on $\Omega$ for which 
the right-hand side is finite (this relation easily follows from 
the elementary inequality $xy \leq x\log x - x + e^y$, $x \geq 0$, $y \in \R$).

In the case where $\Omega = \R^d$ with Lebesgue measure
$\lambda$, and choosing $g(x) = \ep\,|x|^2$, $\ep > 0$, we have in particular
$$
\ep\, \E\,|X|^2 \, \leq \, D(X||Y) + \log \, \E\,e^{\ep\, |Y|^2}.
$$
If $Y$ has a normal distribution, the last expectation is finite for some
$\ep > 0$. Therefore, finiteness of $D(X||Y)$ forces the random vector 
$X$ in $\R^d$ to have a finite second moment. One can now introduce
an affine invariant functional
$$
D(X) \ = \, \inf_{Y \ {\rm normal}} D(X||Y),
$$
where the infimum is running over all absolutely continuous normal distributions 
on $\R^d$. Thus, $D(X)$ represents the Kullback-Leibler distance from the 
distribution of $X$ to the class of all non-degenerate Gaussian measures 
on $\R^d$. It is finite, only if the distribution of $X$ is 
absolutely continuous and has a finite second moment, and 
then this infimum is attained on the normal distribution with 
the same mean $a = \E X$ and covariance matrix $V$ as for $X$ (cf. e.g. \cite{Bi}, Section 10.7).

Our next step is to quantify the properties $(a)-(b)$ from
Proposition 3.1 in terms of $D(X||Z)$, where $Z$ is a standard normal
random vector in $\R^d$. Denote by $\varphi_{a,V}$ the density 
of the normal law with these parameters, that is, let $Y$ have density
$$
\varphi_{a,V}(x) \ = \ \frac{1}{(2\pi)^{d/2} \sqrt{{\rm det}(V)}}\,
\exp\Big\{-\frac{1}{2}\left<V^{-1}(x-a),x-a\right>\Big\}, \qquad x \in \R^d,
$$
so that $D(X) = D(X||Y)$.
By the definition, if $X$ has density $p$, we have
\bee
D(X||Z)
 & = & 
\int_{\R^d} p(x) \log \frac{p(x)}{\varphi(x)}\,dx \ = \
\int_{\R^d} p(x) \log \frac{p(x)}{\varphi_{a,V}(x)}\,dx \\
 & & - \ 
\frac{1}{2}\,\log\, {\rm det}(V) - \frac{1}{2}\,\E\,\left<V^{-1}(X-a),X-a\right> + 
\frac{1}{2}\,\E\,|X|^2.
\ene
Simplifying, we obtain an explicit formula
\begin{eqnarray}\label{(3.1)}
D(X||Z) 
 & = & 
D(X) +  \frac{1}{2}\,|a|^2 + \frac{1}{2}\,
\Big(\log\frac{1}{{\rm det}(V)} + {\rm Tr}(V) - d\Big) \nonumber \\
 & = &
D(X) +  \frac{1}{2}\,|a|^2 + \frac{1}{2}\,\sum_{i=1}^d 
\Big(\log\frac{1}{\sigma_i^2} + \sigma_i^2 - 1\Big),
\end{eqnarray}
where $\sigma_i^2$ are eigenvalues of the matrix $V$ ($\sigma_i > 0$).
Note that all the terms on the right-hand side are non-negative. This allows 
us to control the first two moments of $X$ in terms of $D(X||Z)$. 
In particular, $|a|^2 \leq 2\,D(X||Z)$, so that the closeness of $X$ to $Z$ 
in relative entropy implies the closeness of the means.
To come to a similar conclusion about the covariance matrices, consider the
non-negative convex function
$$
\psi(t) = \log\frac{1}{t} + t - 1, \qquad t>0.
$$
We have $\psi(1) = \psi'(1) = 0$ and $\psi''(t) = \frac{1}{t^2}$.
If $|t-1| \leq 1$, by Taylor's formula about the point $t_0 = 1$
with some point $t_1$ between $t$ and $1$,
$$
\psi(t) = \psi(1) + \psi'(1) (t-1) + \psi''(t_1)\, \frac{(t-1)^2}{2}
 \geq \frac{(t-1)^2}{8}.
$$
For the values $t \geq 2$, we have a linear bound
$\log\frac{1}{t} + t - 1 \geq c(t-1)$ with some constant $0<c<1$.
Namely, write the latter inequality as $\log t \leq (1-c)(t-1)$, i.e.,
$u(s) = \frac{\log(1+s)}{s} \leq 1-c$ for $s \geq 1$.
As easy to check, the function $u(s)$ is decreasing on the whole
positive axis, so $u(s) \leq \log 2$ in $s \geq 1$. Hence, one may
take $c = 1 - \log 2 > \frac{1}{8}$, and thus
$\psi(t) \geq \frac{t-1}{8}$. The two bounds yield
$$
\psi(t) \geq \frac{1}{8}\,\min\{|t - 1|, |t - 1|^2\}, \quad t>0.
$$
Let us summarize.

\begin{lemma}\label{3.2}

Given a random vector $X$ with mean $a$ and covariance matrix $V$ 
with eigenvalues $\sigma_i^2$, we have
$$
D(X||Z) \, \geq \, D(X) + \frac{1}{2}\,|a|^2 + \frac{1}{16}\,\sum_{i=1}^d
\min\big\{|\sigma_i^2 - 1|, (\sigma_i^2 - 1)^2\big\}.
$$
In particular, putting $D = D(X||Z)$, we have
\begin{itemize}
\item[(a)] $|a|^2 \leq 2D$;
\item[(b)] $|\sigma_i^2 - 1| \leq 4\sqrt{D} + 16 D$ \ for all $i \leq d$;
\item[(c)] $|\E\, |X|^2 - d\,| \leq 4d\sqrt{D} + 16d\, D$.
\end{itemize}

\end{lemma}

Here, the closeness of all $\sigma_i^2$ to 1 may also be stated as 
closeness of $V$ to the identity $d \times d$ matrix $I_d$ in the (squared)
Hilbert-Schmidt norm 
$\|V - I_d\|_{\rm HS}^2 = \sum_{i=1}^d (\sigma_i^2 - 1)^2$.
These bounds have an application in the problem 
where one needs to determine whether or not there is convergence
in relative entropy for a sequence of random vectors. 

\begin{corollary}\label{3.3}

Given a sequence of random vectors
$Z_n$ in $\R^d$ with means $a_n$ and covariance matrices $V_n$, the property 
$D(Z_n||Z) \rightarrow 0$ as 
$n \rightarrow \infty$ is equivalent to the next three conditions:
$$ D(Z_n) \rightarrow 0; \quad a_n \rightarrow 0; \quad V_n \rightarrow I_d. $$

\end{corollary}

\begin{proof}[{\bf Proof of Proposition} \ref{3.1}]
First recall that
$$
D(Z_n || Z) \, = \,
-h(Z_n) + \frac{d}{2}\,\log(2\pi) + \frac{1}{2}\,\E\,|Z_n|^2, \qquad
h(Z) = \frac{d}{2}\,\log(2\pi) + \frac{d}{2}.
$$
Hence, if $\E\,|Z_n|^2 \rightarrow d$ like in $(a)$, then
$D(Z_n||Z) \rightarrow 0 \Leftrightarrow h(Z_n) \rightarrow h(Z)$.
To show that the conditions $(a)-(b)$ are sufficient for the convergence
in $D$, denote by $f_n$ the characteristic functions of $Z_n$. 
By the assumption and applying the Plancherel theorem,
$$
\Delta_n \, = \, (2\pi)^{-d/2}\, \|f_n - g\|_2 \ \rightarrow \ 0
$$
as $n \rightarrow \infty$. Define the random vectors $\tilde Z_n = b_n Z_n$, 
where $b_n^2 = d/\E\,|Z_n|^2$ ($b_n > 0$), so that $\E\, |\tilde Z_n|^2 = d$.
They have densities $\tilde p_n(x) = \frac{1}{b_n^d}\,p_n(\frac{x}{b_n})$ 
with characteristic functions
$$
\tilde f_n(t) = \E\,e^{i\left<t,\tilde Z_n\right>} = f_n(b_n t), \qquad
t \in \R^d.
$$
Using $b_n \rightarrow 1$ and applying 
the Plancherel theorem once more together with the triangle inequality
in $L^2$, we then get
\bee
\tilde \Delta_n
 =
 \ (2\pi)^{-d/2}\,
\|\tilde f_n - g\|_2
& = &
\frac{1}{(2\pi b_n)^{d/2}}\, \|f_n(t) - g(t/b_n)\|_2 \\
 & \leq &
\frac{1}{(2\pi b_n)^{d/2}}\, \|f_n(t) - g(t)\|_2 + 
\frac{1}{(2\pi b_n)^{d/2}}\, \|g(t/b_n) - g(t)\|_2 \\
 & = &
\frac{1}{b_n^{d/2}}\, \Delta_n + \frac{1}{(2\pi b_n)^{d/2}}\, \|g(t/b_n) - g(t)\|_2.
\ene
Here, the last norm tends to zero, so, $\tilde \Delta_n \, \rightarrow \, 0$.
We are in position to apply the upper bound \eqref{(2.1)} 
of Proposition \ref{2.1} to $X = \tilde Z_n$ which yields 
$D(\tilde Z_n||Z) \rightarrow 0$ and thus
\be\label{(3.2)}
D(Z_n||Z) \, = \, D(\tilde Z_n||Z) - d\log b_n + \frac{d}{2}\,(b_n^2 - 1) 
\, \rightarrow \, 0.
\en
Conversely, assuming that $D(Z_n||Z)\, \rightarrow \, 0$ and applying 
Corollary \ref{3.3}, we get the property $(a)$. Hence, 
$b_n^2 = d/\E\,|Z_n|^2 \rightarrow 1$, and $D(\tilde Z_n||Z)\, \rightarrow \, 0$
according to the formula \eqref{(3.2)}. By the assumption, $\tilde p_n$ are
uniformly bounded, that is, $\tilde p_n(x) \leq M$ with some constant $M$.
We are in position to apply the lower bound \eqref{(2.2)} which yields 
$\tilde \Delta_n \rightarrow 0$ and therefore
$$ \Delta_n = b_n^{d/2}\,(2\pi)^{-d/2}\, \|\tilde f_n(t) - g(b_n t)\|_2 \leq b_n^{d/2}\, \tilde \Delta_n + b_n^{d/2}\,(2\pi)^{-d/2}\, \|g(t) - g(b_n t)\|_2 \, \rightarrow \, 0. $$
\end{proof}

\section{Proof of Theorems \ref{1.1}-\ref{1.2}}\label{4}

From now on, let the random vectors $Z_n$ be defined as the normalized sums 
according to \eqref{(1.1)}. The proof of Theorem \ref{1.1} is based on 
the following statement obtained in \cite{B-M}. 

\begin{lemma}(\cite[Theorem 1.3]{B-M})\label{4.1}
There exists $T>0$ depending on the distribution of $X_1$ with the following
property. If $f$ is supported on the ball $|t| \leq T$, then the random vectors 
$Z_n$ have continuous densities $p_n$ such that 
\be\label{(4.1)}
\|p_n - \varphi\|_\infty \, = \, 
\sup_x\, |p_n(x) - \varphi(x)| \, \rightarrow 0 \, \quad \mbox{as} \ \ 
n \rightarrow \infty.
\en
If $\beta_3$ is finite, one may take $T = 1/\beta_3$. If $X_1$ has 
a non-lattice distribution, $T$ may be arbitrary.
\end{lemma}

Recall that, in Theorems \ref{1.1}-\ref{1.2} we assume that $\E\,|X|^2 < \infty$, 
which implies $\E\,|Z_n|^2 = \frac{1}{n}\,\E\,|X|^2 + d \rightarrow d$
as $n \rightarrow \infty$. In addition,
the uniform convergence \eqref{(4.1)} is stronger than
\be\label{(4.2)}
\|p_n - \varphi\|_2 \, \rightarrow \, 0 \quad \mbox{as} \ \ n \rightarrow \infty,
\en
since
\bee
\|p_n - \varphi\|_2^2 
 & = & 
\int_{\R^d} (p_n(x) - \varphi(x))^2\,dx \\
 & \leq & 
\|p_n - \varphi\|_\infty \ \int_{\R^d} |p_n(x) - \varphi(x)|\,dx
 \ \leq \ 2\,\|p_n - \varphi\|_\infty.
\ene
By Proposition \ref{3.1}, both properties ensure 
that $D(Z_n||Z) \rightarrow 0$, and we obtain Theorem \ref{1.1}.

Now, let us turn to the Bernoulli case, that is, when $X_1$ has 
a uniform distribution on the discrete cube $\{-1,1\}^d$. 
Theorem \ref{1.2} may slightly be refined in one direction by weakening
the condition \eqref{(1.4)}. As before, $\|t\|$ denotes the distance 
from the point $t \in \R^d$ to the lattice $\pi \Z^d$. 

\begin{theorem}\label{4.2}

Suppose that the characteristic function of $X$
satisfies
\be\label{(4.3)}
f(\pi k) = 0 \quad \mbox{for all} \ \ k \in \Z^d, \ k \neq 0,
\en
together with
\be\label{(4.4)}
\int_{\R^d} \frac{|f(t)|\, |f'(t)|}{\|t\|^{d-1}}\,dt < \infty.
\en
Then we have the entropic CLT, that is, $D(Z_n||Z) \rightarrow 0$ as 
$n \rightarrow \infty$. Conversely, if the entropic CLT holds together with
\be\label{(4.5)}
\int_{\R^d} |f(t)|\,dt < \infty, \quad
\int_{\R^d} \frac{|f'(t)|}{\|t\|^{d-1}}\,dt < \infty,
\en
then $f$ satisfies \eqref{(4.3)}. In this case the uniform local limit
theorem \eqref{(4.1)} takes place.

\end{theorem}

The point of the refinement is that \eqref{(4.4)} is weaker than \eqref{(4.5)}, 
which is exactly the condition \eqref{(1.4)} in Theorem \ref{1.2}. 
In dimension $d=1$, \eqref{(4.4)} is fulfilled whenever $f$ and $f'$ are 
in $L^2$ (by Cauchy's inequality), that is, when the density $p$ of 
the random variable $X$ satisfies
$$
\int_{-\infty}^\infty (1+x^2)\,p(x)^2\,dx < \infty
$$
(which holds automatically, if $p$ is bounded). If $d \geq 2$, \eqref{(4.4)} 
is fulfilled under the decay assumptions \eqref{(1.6)} with a weaker parameter 
constraint $\alpha > \frac{1}{2}$. This is the case, for example, where $X$ 
is uniformly distributed in the (solid) cube $[-1,1]^d$, while \eqref{(4.5)} 
does not hold. In \cite{B-M}, it was shown that the properties 
\eqref{(4.3)}-\eqref{(4.4)} imply the $L^2$-convergence of densities 
\eqref{(4.2)}, while \eqref{(4.3)} together with a stronger assumption 
\eqref{(4.5)} leads to the uniform convergence \eqref{(4.1)}. Hence, 
we can apply Proposition \ref{3.1} to conclude that $D(Z_n||Z) \rightarrow 0$. 
It was also shown there that the property \eqref{(4.3)} is fulfilled under the 
$L^2$-convergence \eqref{(4.2)}. In order to arrive at a similar conclusion 
under an apriori weaker entropic CLT, we involve the assumption \eqref{(4.5)} 
and prove here:

\begin{lemma}\label{4.3}

Suppose that $X_1$ has a uniform distribution
on the discrete cube $\{-1,1\}^d$. If the condition \eqref{(4.5)} 
is fulfilled, then $Z_n$ have uniformly bounded densities $p_n$.

\end{lemma}

Having this assertion, we therefore complete the proof of Theorem \ref{4.2}
and of Theorem \ref{1.2} by appealing to Proposition \ref{3.1} once more.

\begin{proof}[{\bf Proof of Lemma} \ref{4.3}] 
Put $v(t) = \cos(t_1) \dots \cos(t_d)$ for $t = (t_1,\dots,t_d) \in \R^d$.
By the assumption \eqref{(4.5)}, the characteristic functions 
$$
f_n(t) = f\Big(\frac{t}{\sqrt{n}}\Big)\,v^n\Big(\frac{t}{\sqrt{n}}\Big)
$$ 
are integrable. Hence, $Z_n$ have continuous densities 
given by the Fourier inversion formula
\be\label{(4.6)}
p_n(x) \, = \, \frac{1}{(2\pi)^d} \int_{\R^d} e^{-i\left<t,x\right>} f_n(t)\,dt.
\en
Let us partition $\R^d$ into the cubes $Q_k = Q + \pi k$, 
$Q = [-\frac{\pi}{2}, \frac{\pi}{2}]^d$, $k \in \Z^d$, so that 
$\|t\| = |t - \pi k|$ for $t \in Q_k$. Splitting the integration in \eqref{(4.6)},
we can write
$$
p_n(x) \, = \, \frac{1}{(2\pi)^d} \sum_{k \in \Z^d} I_{n,k}(x), \qquad
I_{n,k}(x) \, = \, n^{d/2}
\int_{Q_k} e^{-i\left<t,x\right>\sqrt{n}}\, f(t) v^n(t)\,dt.
$$
Putting $w_k(t) = f(\pi k + t)$ and using the periodicity of the cosine 
function together with the bound 
$0 \leq \cos(u) \leq e^{-u^2/2}$ for $|u| \leq \frac{\pi}{2}$, we have
$$
|I_{n,k}(x)| \, \leq \, n^{d/2} J_{n,k}, \quad 
J_{n,k} \, = \, \int_Q |w_k(t)|\,e^{-n |t|^2/2} \,dt.
$$
By Taylor's formula,
\be\label{(4.7)}
|f(\pi k + t) - f(\pi k)| \, \leq \, |t|
\int_0^1 |f'(\pi k + \xi t)|\, d\xi, \qquad t \in \R^d.
\en
Hence, changing the variable $\xi t = s$, we get
\bee
\int_Q |f(\pi k + t) - f(\pi k)|\, dt
 & \leq &
\int_0^1 \int_Q |f'(\pi k + \xi t)|\,|t|\, d\xi\, dt \\
 & = &
\int_Q \bigg[|f'(\pi k + s)|\,|s|
\int_{\frac{2}{\pi}\,\|s\|_\infty}^1 \frac{d\xi}{\xi^{d+1}} \bigg]\, ds
 \, \leq \,
c_d \int_Q \frac{|f'(\pi k + s)|}{|s|^{d-1}}\, ds
\ene
with some constant $c_d$ depending on $d$ only, where 
$\|s\|_\infty = \max_k |s_k|$ for $s = (s_1, \dots, s_d) \in \R^d$. Hence
$$
\pi^d\ |w_k(0)| \, = \, 
\pi^d\ |f(\pi k)| \, \leq \, \int_{Q_k} |f(t)|\, dt + 
c_d \int_{Q_k} \frac{|f'(t)|}{\|t\|^{d-1}}\, dt.
$$
The next summation over all $k$ leads to
\be\label{(4.8)}
\sum_{k \in \Z^d} |w_k(0)| \, = \, 
\sum_{k \in \Z^d} |f(\pi k)| \, \leq \, \frac{1}{\pi^d}
\int_{\R^d} |f(t)|\, dt + 
\frac{c_d}{\pi^d} \int_{\R^d} \frac{|f'(t)|}{\|t\|^{d-1}}\, dt \, < \, \infty,
\en
where we applied the assumption \eqref{(4.5)}.
Put
$$
\widetilde J_{n,k} \, = \, \int_Q (|w_k(t)| - |w_k(0)|)\,e^{-n |t|^2/2} \,dt.
$$
By \eqref{(4.7)},
$$
|w_k(t)| \, \leq \, |w_k(0)| + |t| \int_0^1 |w_k'(\xi t)|\, d\xi.
$$
Hence, again changing the variable $\xi t = s$, and then 
$\xi = \sqrt{n}\,|s|\,\frac{1}{u}$, we get
\bee
\widetilde J_{n,k}
 & \leq &
\int_Q \int_0^1 |t|\, |w_k'(\xi t)|\, e^{-n |t|^2/2} \,dt\,d\xi \\
 & = &
\int_Q |w_k'(s)|\, |s| \ 
\bigg[\int_{\frac{2}{\pi} \|s\|_{\infty}}^1 \, \xi^{-d-1}\, 
e^{-n|s|^2/2\xi^2} \,d \xi \bigg]\,ds \\
 & \leq &
n^{-d/2}
\int_Q |w_k'(s)|\, |s|^{-(d-1)} \ 
\bigg[\int_{|s|\sqrt{n}}^\infty \,
u^{d-1}\, e^{-u^2/2} \,du\bigg]\,ds \\
 & \leq & 
c_d\, n^{-d/2}
\int_Q \frac{|w_k'(s)|}{|s|^{d-1}}\, e^{-n |s|^2/2}\,ds
\ene
with some constant $c_d$ depending on the dimension, only. 
Performing summation over all $k$, we get
$$
n^{d/2} \ \sum_{k \in \Z^d} \widetilde J_{n,k} \, \leq \, c_d \int_{\R^d} 
\frac{|f'(t)|}{\|t\|^{d-1}}\,e^{-n\,\|t\|^2/2} \,dt \, \leq \, 
c_d \int_{\R^d} \frac{|f'(t)|}{\|t\|^{d-1}} \,dt.
$$
Due to \eqref{(4.8)}, with some other $d$-dependent constants
$$
n^{d/2} \ \sum_{k \in \Z^d} J_{n,k} \, \leq \, c_d \int_{\R^d} |f(t)|\,dt + 
c_d \int_{\R^d} \frac{|f'(t)|}{\|t\|^{d-1}} \,dt \, < \, \infty,
$$
and thus $\sum_{k \in \Z^d} |I_{n,k}(x)|$ is bounded by a constant 
which does not depend on $x$.
\end{proof}

\vskip5mm
\noindent
{\bf Remark 4.4.}
To better realize the meaning of Theorem \ref{1.1},
let us also comment on the relationship between the entropic and
transport CLT's. Given two random vectors $X$ and $Y$ in $\R^d$ 
with distributions $\mu$ and $\nu$ respectively, the (quadratic)
Kantorovich distance is defined as
$$
W_2(\mu,\nu) \, = \, W_2(X,Y) \, = \,
\inf_\lambda \bigg(\int_{\R^d} \int_{\R^d} |x-y|^2\,d\lambda(x,y)\bigg)^{1/2}
$$
where the infimum is running over all (Borel) probability measures
$\lambda$ on $\R^d \times \R^d$ with marginals $\mu$ and $\nu$. It 
represents a metric in the space $M_2(\R^d)$ of all probability measures
on $\R^d$ with finite second moment, which is closely related to the 
weak topology. More precisely, given a sequence $\mu_n$ and a ``point" 
$\nu$ in $M_2(\R^d)$, the convergence $W_2(\mu_n,\nu) \rightarrow 0$ 
holds true as $n \rightarrow \infty$ if and only if $\mu_n$ are weakly 
convergent to $\nu$, that is, 
$$
\int_{\R^d} u(x)\,d\mu_n(x) \rightarrow \int_{\R^d} u(x)\,d\nu(x)
$$
for any bounded continuous function $u$ on $\R^d$, and
$\int_{\R^d} |x|^2\,d\mu_n(x) \rightarrow \int_{\R^d} |x|^2\,d\nu(x)$
(cf. e.g. \cite{V}, p. 212). 

When $\nu$ is the standard Gaussian measure on $\R^d$, the relationship
of $W_2$ with relative entropy was emphasized by Talagrand \cite{T} who
showed that
$$
W_2^2(X,Z) \leq 2\,D(X||Z)
$$
holding for any random vector $X$ in $\R^d$ with $Z$ a standard normal
random vector.
Returning to the setting of Theorem \ref{1.1}, define the normalized sums
$$
Z_n' = Z_n - \frac{1}{\sqrt{n}}\,X =
\frac{1}{\sqrt{n}}\,(X_1 + \dots + X_n).
$$
By the classical CLT, the distributions $\mu_n'$ of $Z_n'$ are weakly 
convergent to the Gaussian limit $\nu$. Since also
$\E\,|Z_n|^2 = \E\,|Z|^2 = d$, the above characterization of the convergence
in the space $M_2(\R^d)$ ensures that $W_2(\mu_n',\nu) \rightarrow 0$,
which is a transport CLT. A similar conclusion can also be made
on the basis of Theorem \ref{1.1}. Indeed, choose for $f$
a characteristic function supported on a suitable small ball $|t| \leq T$,
so that $D(Z_n||Z) \rightarrow 0$, by \eqref{(1.3)}. Applying the Talagrand
transport-entropy inequality, we get
$$
W_2^2(Z_n',Z) \, \leq \, 2W_2^2(Z_n,Z) + \frac{2}{n}\,\E\,|X|^2 \, \leq \,
4\,D(Z_n||Z) + \frac{2}{n}\,\E\,|X|^2 \ \rightarrow \ 0.
$$

A similar approach was used in \cite{Bo} to study the rate of convergence
in the one-dimensional transport CLT under the 4-th moment assumption.

\section{Entropy bounds}\label{5}

Let $(X_n)_{n \geq 1}$ be a sequence of integer valued random vectors 
in $\R^d$, and let $X$ be a continuous random vector in $\R^d$ with 
finite second moment, independent of this sequence.
As before, we define the normalized sums
$$
Z_n = \frac{1}{\sqrt{n}}\,(X + X_1 + \dots + X_n).
$$
As is well-known, when the second moment $\E\,|U|^2$ of a continuous 
random vector $U$ in $\R^d$ is fixed, its entropy is maximized on the normal 
distribution with the same second moment (cf. e.g. \cite{C-T}). In the case
of independent and isotropic $X_n$'s, we have
$\E\,|Z_n|^2 = \frac{1}{n}\,\E\,|X|^2 + d \rightarrow d$ as 
$n \rightarrow \infty$. Hence $\limsup_{n\rightarrow \infty} h(Z_n) \leq h(Z)$, 
where $Z$ is a standard normal random vector in $\R^d$. The argument 
to derive a similar bound 
$\limsup_{n\rightarrow \infty} h(Z_n) \leq h(Z) + h(X)$ is based on two
elementary lemmas, which involve the discrete Shannon entropy
$$
H(Y) = - \sum_k p_k \log p_k.
$$
Here, $Y$ is a discrete random vector taking at most countably many values, 
say $y_k$, with probabilities $p_k$ respectively.

\begin{lemma}\label{5.1}

Let $X$ be a continuous random vector, and let $Y$ be a discrete random 
vector independent of $X$, both with values in the Euclidean space $\R^d$. 
Then
$$
h(X+Y) \leq h(X) + H(Y). 
$$

\end{lemma}

Lemma \ref{5.1} can be derived implicitly from the ideas of \cite{N-P-S} 
about the entropy of mixtures of discrete and continuous random variables. 
An explicit statement appears in \cite[Lemma 11.2]{W-M} (see also \cite{M-T-B-S}). 
We include a proof for completeness:

\begin{proof}
Denote by $p$ the density of $X$ and let $p_k = P\{Y=y_k\}$
for some finite or infinite sequence $y_k$.
Since $X$ and $Y$ are independent, $X+Y$ has density
$$
q(z) = \sum_k p_k p(z-y_k). 
$$
We use the convention $u \log(u) = 0$ if $u=0$. Note that, if $p(z - y_k) = 0$, 
then
$$ 
p_k p(z-y_k) \log \sum_i p_i p(z-y_i) = 0 = p_k p(z-y_k) \log(p_k p(z-y_k)),
$$
while in the case $p(z-y_k) > 0$, we have
\bee
p_k p(z-y_k) \log \sum_i p_i p(z-y_i) 
 & = & 
p_k p(z-y_k) \log\Big(p_k p(z-y_k) + \sum_{i \neq k} p_i p(z-y_i)\Big) \\ 
 & & \hskip-20mm = \
p_k p(z-y_k)\, \bigg[\log(p_k p(z-y_k)) + 
\log\Big(1 + \frac{\sum_{i \neq k} p_i p(z-y_i)}{p_k p(z-y_k)}\Big) \bigg] \\ 
 & & \hskip-20mm \geq \
p_k p(z-y_k)\, \log(p_k p(z-y_k)).
\ene
Hence, for all $z$,
$$ 
p_k p(z-y_k)\, \log\, \sum_i p_i p(z-y_i) \, \geq \, 
p_k p(z-y_k)\, \log(p_k p(z-y_k)). 
$$
We may therefore conclude that
\bee
h(X+Y) 
 & = & 
- \int_{\R^d} q(z) \log q(z)\,dz \\ 
 & = & 
- \sum_k \int_{\R^d} p_k p(z-y_k)\, \log\, \sum_i p_i p(z-y_i) \, dz \\ 
 & \leq & 
- \sum_k \int_{\R^d} p_k p(z-y_k) \log(p_k p(z-y_k)) \, dz \\ 
 & = & 
- \sum_k p_k\, \bigg(\int_{\R^d} p(z-y_k) \log p_k \, dz + 
\int_{\R^d} p(z-y_k) \log p(z-y_k) \, dz \bigg) \\
 & = &
h(X) + H(Y).
\ene
\end{proof}

Let us note that a recent sharpening of Lemma \ref{5.1} appears in \cite[Theorem III.1]{M-M-S}, where it is shown that
$$ h(X + Y ) \leq h(X|Y) + T H(Y), $$
where $h(X|Y)$ is the conditional entropy, reducing to $h(X)$ on independence, and $T$ is the supremum of the total variation of the conditional densities from their ``mixture complements'', necessarily $T \leq 1$.

The following lemma is standard and has been used in several applications 
(see \cite{M}):

\begin{lemma}\label{5.2}

For any integer valued random variable $Y$ with finite second moment,
\be\label{(5.2)}
H(Y) \, \leq \, 
\frac{1}{2} \log\Big(2\pi e \Big(\Var(Y) + \frac{1}{12}\Big)\Big). 
\en

\end{lemma}

The proof of Lemma \ref{5.2}, that we include for completeness, also combines 
both discrete and differential entropy:

\begin{proof}
Put $p_k = \P\{Y=k\}$, $k \in \Z$. Consider a continuous random variable 
$\widetilde Y$ with density $q$ defined to be
$$ 
q(x) = p_k \quad \mbox{if $x \in (k-\frac{1}{2}, k+\frac{1}{2})$}. 
$$
In other words,
$$ 
q(x) = \sum_k p_k 1_{(k-\frac{1}{2}, k+\frac{1}{2})}(x), \quad x \in \R. 
$$
Note that
$$ 
\E \widetilde Y = \sum_k p_k \int_{k-\frac{1}{2}}^{k+\frac{1}{2}} x \, dx = 
\sum_k \frac{p_k }{2} 
\Big(\Big(k+\frac{1}{2}\Big)^2 - \Big(k-\frac{1}{2}\Big)^2 \Big) = 
\sum_k k p_k = \E Y
$$
and similarly
$$ 
\E \widetilde Y^2 = 
\sum_k p_k \int_{k-\frac{1}{2}}^{k+\frac{1}{2}} x^2 \, dx = 
\E Y^2 + \frac{1}{12}. 
$$
Hence
$
\Var(\widetilde Y) = \Var(Y) + \frac{1}{12}. 
$
Also,
\bee
h(\widetilde Y)
 & = &
- \int_{-\infty}^{\infty} \sum_k p_k 1_{(k-\frac{1}{2}, k+\frac{1}{2})}(x)\, 
\log\, \sum_j p_j 1_{(j-\frac{1}{2}, j+\frac{1}{2})}(x) \, dx \\
 & = &
- \sum_k p_k \int_{k-\frac{1}{2}}^{k+\frac{1}{2}} \log p_k \, dx \ = \ H(Y).
\ene
Now, since Gaussian distributions maximize the differential entropy 
for a fixed variance, we conclude that
$$ 
H(Y) \, = \, h(\widetilde Y) \, \leq \, 
\frac{1}{2} \log\big(2 \pi e\, \Var(\widetilde Y)\big) \, = \,
\frac{1}{2} \log\Big(2 \pi e \Big(\Var(Y) + \frac{1}{12}\Big)\Big). 
$$
\end{proof}

We are now prepared to establish Theorem \ref{1.3}, in fact under somewhat weaker assumptions.

\begin{theorem}\label{5.3}
Given a sequence $X_n = (X_{n,1},\dots,X_{n,d})$ of random vectors 
with values in $\Z^d$, independent of $X$, assume that for each $k \leq d$, 
the components $X_{n,k}$, $n \geq 1$, are uncorrelated
and have variance one. Then,
$$ 
\limsup_{n \to \infty} \, h(Z_n) \, \leq \, h(X) + h(Z). 
$$
\end{theorem}

\begin{proof}
Putting $S_n = X_1 + \dots + X_n$ and applying Lemma \ref{5.1}, we get
\begin{eqnarray*}
h(Z_n) \ = \
h\Big( \frac{X + S_n}{\sqrt{n}}\Big) 
 & = & 
h(X + S_n) - \frac{d}{2} \log n \\ 
 & \leq & 
h(X) + H(S_n) - \frac{d}{2} \log n.
\end{eqnarray*}
Note that
$$
S_n = (S_{n,1},\dots,S_{n,d}), \qquad 
S_{n,k} = X_{1,k} + \dots + X_{n,k} \quad (1 \leq k \leq d).
$$
By the well-known subadditivity of entropy along components of 
a random vector (an abstract property on product spaces which is 
irrelevant to the independence assumption, cf. e.g. \cite{L}), we have
$$
H(S_n) \, \leq \, H(S_{n,1}) + \dots + H(S_{n,d}).
$$
Here, the entropy functional on the left is applied to the $d$-dimensional
random vector, while on the right-hand side of this inequality we deal 
with one-dimensional entropies. For each $k \leq d$, the $k$-th component
$S_{n,k}$ of the random vector $S_n$ represents the sum of $n$ uncorrelated
integer valued random variables with variance one, so that
$\Var(S_{n,k}) = n$. Hence, by \eqref{(5.2)} applied to $Y = S_{n,k}$, 
we have
$$
H(S_{n,k}) \, \leq \,
\frac{1}{2} \log\Big(2 \pi e \Big(n + \frac{1}{12}\Big)\Big) \, = \, 
\frac{1}{2}\, \log(2 \pi e n) + O(1/n),
$$
and therefore
$$
H(S_n) \, \leq \, \frac{d}{2}\, \log(2 \pi e n) + O(1/n).
$$
We conclude that
$$ 
\limsup_{n \to \infty} \, \, h(Z_n) \, \leq \, 
h(X) + \frac{d}{2} \log(2 \pi e) \, = \, h(X) + h(Z).
$$
\end{proof}

{\bf Acknowledgements.} 
Research of S.B. was partially supported by the Simons Foundation and NSF
grant DMS-1855575. The authors are grateful to both referees for useful comments.

\end{document}